\documentclass[11pt]{amsart}
\usepackage{eurosym}
\usepackage{amssymb}
\usepackage{amsmath}
\usepackage{amsfonts}
\usepackage{graphicx,color}
\usepackage{xcolor}
\usepackage{bbm}

\newtheorem{theorem}{Theorem}
\theoremstyle{plain}

\newtheorem{proposition}{Proposition}
\newtheorem{remark}{Remark}

\numberwithin{equation}{section}

\begin{document}
\title[Singular quasilinear elliptic problems]{nodal solutions with
synchronous sign changing components and Constant sign solutions for
singular Gierer-Meinhardt type system}
\author[A. Moussaoui]{Abdelkrim Moussaoui}
\address{Applied Mathematics Laboratory, Faculty of Exact Sciences, \\
and Biology Department, Faculty of Natural and Life Sciences \\
A. Mira Bejaia University, Targa Ouzemour, 06000 Bejaia, Algeria}
\email{abdelkrim.moussaoui@univ-bejaia.dz}

\begin{abstract}
We establish the existence of three solutions for singular semilinear elliptic system, two of which are of opposite constant-sign. Under a strong singularity effect, the third solution is nodal with synchronous sign components. The approach combines sub-supersolutions method and Leray-Schauder topological degree involving perturbation argument.
\end{abstract}

\maketitle

\section{Introduction}

\label{S1}

Let $\Omega $ is a bounded domain in $%
%TCIMACRO{\U{211d} }%
%BeginExpansion
\mathbb{R}
%EndExpansion
^{N}$ $\left( N\geq 2\right) $ with a smooth boundary $\partial \Omega $. We
consider the following system of semilinear elliptic equations%
\begin{equation*}
(\mathrm{P})\qquad \left\{ 
\begin{array}{ll}
-\Delta u+u=sgn(v)\frac{|u|^{\alpha _{1}}}{|v|^{\beta _{1}}} & \text{in }%
\Omega \\ 
-\Delta v+v=sgn(u)\frac{|u|^{\alpha _{2}}}{|v|^{\beta _{2}}} & \text{in }%
\Omega \\ 
u,v=0\text{ } & \text{on }\partial \Omega ,%
\end{array}%
\right.
\end{equation*}%
where $\Delta $ stands for the Laplacian differential operator on $\mathcal{H%
}_{0}^{1}(\Omega ),$ $sgn(\cdot )$ denotes the sign function and the
exponents\ $\alpha _{i}\in (-1,1)$ and $\beta _{i}\in (0,1)$\ satisfy the
following condition 
\begin{equation}
\alpha _{i}+\beta _{i}<1\text{ \ and \ }0>\alpha _{i}-\beta _{i}>-1,\text{
for }i=1,2\text{.}  \label{alpha}
\end{equation}

System $(\mathrm{P})$ exhibits a singularity that, without loss of
generality, is located at zero. This make difficult any study attendant to
sign properties of solutions for $(\mathrm{P}),$ especially those that
change sign since they inevitably pass through the singularity. From a
structural perspective, $(\mathrm{P})$ is closely related to
Gierer-Meinhardt system which is originally arose in studies of biological
pattern formation, describing the activator-inhibitor coupled behavior for
many systems in cell biology and physiology \cite{GMsyst, K, Mein}. It is
characterized by 
\begin{equation}
\left\{ 
\begin{array}{ll}
-d_{1}\Delta u+a_{1}u=\frac{u^{\alpha _{1}}}{v^{\beta _{1}}} & \text{in }
\Omega \\ 
-d_{2}\Delta v+a_{2}v=\frac{u^{\alpha _{2}}}{v^{\beta _{2}}} & \text{in }
\Omega%
\end{array}%
\right.  \label{a}
\end{equation}
subject to Neumann boundary conditions $\frac{\partial u}{\partial \eta }=%
\frac{\partial v}{\partial \eta }=0$ on$\;\partial \Omega ,$ where $u$ and $%
v $ represent the scaled activator and inhibitor concentrations,
respectively, $d_{1},d_{2}$ are diffusion coefficients with $d_{1}\ll d_{2}$
and the exponents $\alpha _{i},\beta _{i}\in 
%TCIMACRO{\U{211d} }%
%BeginExpansion
\mathbb{R}
%EndExpansion
$ satisfy the relations 
\begin{equation*}
\beta _{1}\alpha _{2}<\left( 1-\alpha _{1}\right) \left( 1-\beta _{2}\right) 
\text{ \ with \ }\alpha _{i}\geq 0\geq \beta _{i},\text{ }i=1,2.
\end{equation*}%
Depending on boundedness of the diffusive coefficient $d_{2}$, existence,
stability and/or dynamics of spike solutions are widely studied for system (%
\ref{a}). We refer to \cite{GG, GWW, Ni-Takagi, Ni-Takagi2, W1} when $%
d_{2}\rightarrow +\infty $ while in the bounded case $d_{2}<+\infty ,$ we
quote \cite{JN,LPS, NI-Takagi-Yanagida, WW, Wei2}. The case $d_{1}=d_{2}=1$
for Neumann system (\ref{a}) is examined recently in \cite{M2} establishing
existence of at least three positive solutions. In whole space $\Omega =%
%TCIMACRO{\U{211d} }%
%BeginExpansion
\mathbb{R}
%EndExpansion
^{N}$, existence, uniqueness as well as structural properties of solutions
for Gierer-Meinhardt type systems are studied in \cite{MKT} for $N\geq 3$,
in \cite{DKC, DKW} for $N=1$,$2$, and in \cite{KR, KWY}, when $N=3$.

Gierer-Meinhardt system (\ref{a}) has interesting and challenging
mathematical properties, especially with Dirichlet boundary conditions when
the nonlinear terms become singular near the boundary. In this context,
numerous papers are devoted to study quantitative and qualitative properties
of solutions for system $(\mathrm{P})$ subject to Dirichlet boundary
conditions $u,v=0$ on $\partial \Omega $. A long this direction, among many
other interesting works, we refer interested reader to \cite{CM, GR, PSW}
and the references therein.

The passage from (\ref{a}) to $(\mathrm{P})$ summoned up a dependence of the
nonlinearities on sign of components of solutions. Admittedly, this further
complicates our study since the structure of system $(\mathrm{P})$ will
switch from the cooperative case to the competitive one, depending on the
sign of solutions considered (see, e.g., \cite{AM, DM1, GHM, KM, MM3, MM2,
MM1}). However, the involvement of the sign of solutions components promotes
the emergence of nonpositive solutions for $(\mathrm{P})$, such as nodal
solutions, that have been rarely studied for singular systems. By
definition, a nodal solution is neither positive nor negative. Thus, in the
scalar case, it should necessarily be a sign changing function. However, in
the context of system $(\mathrm{P})$, the concept of a nodal solution is
more nuanced since it incorporates several types of solutions depending on
the sign of their components. Actually, \cite{M} is the only paper that has
considered this issue for singular systems. The existence of nodal solutions
is established for a class of semilinear singular system by means of the
trapping region formed by specific sign changing sub-supersolutions pairs.
Exploiting spectral properties of Laplacian operator as well as adequate
truncation, in \cite{M}, it is established that nodal solutions vanish on
negligible sets. This is an essential point enabling nodal solutions
investigation for singular problems. Hence, by a solution of problem $(%
\mathrm{P})$ we mean $(u,v)\in \mathcal{H}_{0}^{1}(\Omega )\times \mathcal{H}%
_{0}^{1}(\Omega )$ such that $u$ and/or $v$ vanish on zero measure sets and 
\begin{eqnarray*}
\int_{\Omega }(\nabla u\nabla \varphi _{1}+u\varphi _{1})\ \mathrm{d}x
&=&\int_{\Omega }sgn(v)\frac{|u|^{\alpha _{1}}}{|v|^{\beta _{1}}}\varphi
_{1}\ \mathrm{d}x, \\
\int_{\Omega }(\nabla v\nabla \varphi _{2}+v\varphi _{2})\ \mathrm{d}x
&=&\int_{\Omega }sgn(u)\frac{|u|^{\alpha _{2}}}{|v|^{\beta _{2}}}\varphi
_{2}\ \mathrm{d}x,
\end{eqnarray*}%
\ for all $\varphi _{i}\in \mathcal{H}_{0}^{1}(\Omega ),$ for $i=1,2,$
provided the integrals in the right-hand side of the above identities exist.

The aim of this work is to establish multiplicity result for singular system 
$(\mathrm{P})$ with a precise sign information. We provide three solutions
for system $(\mathrm{P}),$ two of which are of opposite constant-sign. The
sign property of the third solution is closely related to the structure of
system $(\mathrm{P})$ which, in turn, depends both on sign of exponents $%
\alpha _{i},\beta _{i}$ and on sign of components of solutions of $(\mathrm{P%
})$. Specifically, when $\alpha _{i}\leq 0$, beside the uniqueness of the
positive and negative solutions, we show that the third solution is nodal
with synchronous sign changing components. According to our knowledge, this
topic is a novelty. Nodal solutions with such property have never been
discussed for systems, whether singular or regular (without singularities),
and even for those with a variational structure.

The main result is stated as follows.

\begin{theorem}
\label{T2}Under assumption (\ref{alpha}), problem $(\mathrm{P})$ has at
least three nontrivial solutions: $(u_{+},v_{+})\in int\mathcal{C}_{+}^{1}(%
\overline{\Omega }))\times int\mathcal{C}_{+}^{1}(\overline{\Omega }),$ $%
(u_{-},v_{-})\in -int\mathcal{C}_{+}^{1}(\overline{\Omega })\times -int%
\mathcal{C}_{+}^{1}(\overline{\Omega })$ and $(u_{\ast },v_{\ast })\in 
\mathcal{H}_{0}^{1}(\Omega )\times \mathcal{H}_{0}^{1}(\Omega )$. If $%
\alpha_{i}\leq 0$ then the opposite constant sign solutions $(u_{+},v_{+})$
and $(u_{-},v_{-})$ are unique and the third solution $(u_{\ast },v_{\ast })$
is nodal with synchronous sign components, that is, 
\begin{equation}
u_{\ast }v_{\ast }>0\text{ \ a.e. in }\Omega \text{.}  \label{40}
\end{equation}
\end{theorem}

The proof of Theorem \ref{T2} combines sub-supersolution method and
topological degree theory. By a choice of suitable functions with an
adjustment of adequate constants, we construct two opposite constant sign
sub-supersolutions pairs on the basis of which positive and negative
rectangles are formed. The latter provide a localisation of a positive and
negative solutions $(u_{+},v_{+})$ and $(u_{-},v_{-})$ of $(\mathrm{P})$,
whose existence is derived from the sub-supersolutions Theorem for singular
systems in \cite[Theorem 2]{KM}. When $\alpha _{i}\leq 0$ in (\ref{alpha}),
the uniqueness of solutions $(u_{+},v_{+})$ and $(u_{-},v_{-})$ is
established by monotonicity argument.

The third solution of $(\mathrm{P})$ is obtained via topological degree
theory. It is located in the area between the positive and the negative
rectangles. This is achieved by introducing first a parameter $\varepsilon>0 
$ in $(\mathrm{P}),$ thus producing a regularized system $(\mathrm{P}%
_{\varepsilon })$ whose study is relevant for problem $(\mathrm{P})$. Then,
we prove that the degree on a ball $\mathcal{B}_{R_{\varepsilon }}$,
encompassing all potential solutions of $(\mathrm{P}_{\varepsilon }),$ is $0$
while the degree in $\mathcal{B}_{R_{\varepsilon }},$ but excluding the area
located between the aforementioned positive and negative rectangles, is
equal to $1$. By the excision property of Leray-Schauder degree, this leads
to the existence of a nontrivial solution $(u_{\varepsilon
},v_{\varepsilon}) $ for $(\mathrm{P}_{\varepsilon }).$ Here, it is
important to note that, unlike \cite{DM1, MMP, M2}, the independence of the
radius $R_{\varepsilon }$ on $\varepsilon $ is not required. Then, through a
priori estimates, dominated convergence Theorem as well as $S_{+}$-property
of the negative Laplacian, we may pass to the limit as $\varepsilon
\rightarrow 0$ in $(\mathrm{P}_{\varepsilon })$. This leads to a solution $%
(u_{\ast },v_{\ast })$ of $(\mathrm{P})$ which, according to its
localization, does not coincide with the above mentioned solutions $%
(u_{+},v_{+})$ and $(u_{-},v_{-})$. Thus, $(u_{\ast },v_{\ast })$ is a third
solution of $(\mathrm{P})$. Furthermore, when $\alpha _{i}\leq 0$ in (\ref%
{alpha}), the uniqueness of the aforementioned constant-sign solutions
forces $(u_{\ast },v_{\ast })$ to be nodal in the sens that both components $%
u_{\ast }$ and $v_{\ast }$ are nontrivial and at least are not of the same
constant sign. However, exploiting the sign coupled property of system $(%
\mathrm{P})$, we conclude that $u_{\ast }$ and $v_{\ast }$ are both of
synchronous sign changing.

The rest of this article is organized as follows. Section \ref{S2} deals
with existence of solutions for regularized system $(\mathrm{P}%
_{\varepsilon})$ while Section \ref{S3} provides multiplicity result for
system $(\mathrm{P})$.

\section{An auxiliary system}

\label{S2}

In the sequel, the Banach spaces $\mathcal{H}_{0}^{1}(\Omega )$ and $%
L^{2}(\Omega )$ are equipped with the usual norms $\Vert \cdot \Vert _{1,2}$
and $\Vert \cdot \Vert _{2}$, respectively. We also utilize the H\"{o}lder
spaces $C^{1}(\overline{\Omega })$ and $C^{1,\tau }(\overline{\Omega }),$ $%
\tau \in (0,1)$ as well as the order cone $\mathcal{C}_{+}^{1}(\overline{%
\Omega })=\{w\in C^{1}(\overline{\Omega }):w(x)\geq 0$ for all $x\in 
\overline{\Omega }\},$ which has a non-empty interior described as follows:%
\begin{equation*}
int\mathcal{C}_{+}^{1}(\overline{\Omega })=\{w\in \mathcal{C}_{+}^{1}(%
\overline{\Omega }):w(x)>0\text{ for all }x\in \overline{\Omega }\}.
\end{equation*}

Hereafter, we denote by $d(x)$ the distance from a point $x\in \overline{%
\Omega }$ to the boundary $\partial \Omega $, where $\overline{\Omega }%
=\Omega \cup \partial \Omega $ is the closure of $\Omega \subset 
%TCIMACRO{\U{211d} }%
%BeginExpansion
\mathbb{R}
%EndExpansion
^{N}$. For $w_{1},w_{2}\in \mathcal{C}^{1}(\overline{\Omega })$, the
notation $w_{1}\ll w_{2}$ means that 
\begin{equation*}
\begin{array}{c}
w_{1}(x)<w_{2}(x)\,\,\,\forall x\in \Omega \,\,\,\mbox{and}\,\,\,\frac{%
\partial w_{2}}{\partial \eta }<\frac{\partial w_{1}}{\partial \eta }\,\,\,%
\mbox{on}\,\,\,\partial \Omega ,%
\end{array}%
\end{equation*}%
$\eta $ is the outward normal to $\partial \Omega $.

Let $y_{i},z_{i}\in int\mathcal{C}_{+}^{1}(\overline{\Omega })$ ($i=1,2$) be
the unique solutions of the Dirichlet problems 
\begin{equation}
-\Delta y_{i}(x)+y_{i}(x)=d(x)^{\alpha _{i}-\beta _{i}}\text{ in }\Omega ,%
\text{ \ }y_{i}=0\text{ on }\partial \Omega ,  \label{20}
\end{equation}%
\begin{equation}
-\Delta z_{i}(x)+z_{i}(x)=\left\{ 
\begin{array}{ll}
d(x)^{\alpha _{i}-\beta _{i}} & \text{in \ }\Omega \backslash \overline{%
\Omega }_{\delta }, \\ 
-1 & \text{in \ }\Omega _{\delta },%
\end{array}%
\right. ,\text{ \ }z_{i}=0\text{ \ on }\partial \Omega ,  \label{22}
\end{equation}%
which is known to satisfy 
\begin{equation}
c^{-1}d(x)\leq z_{i}(x)\leq y_{i}(x)\leq cd(x)\text{ in }\Omega ,  \label{21}
\end{equation}%
where $c>1$ is a constant and 
\begin{equation*}
\Omega _{\delta }=\left\{ x\in \Omega :d(x)<\delta \right\} ,
\end{equation*}%
with a fixed $\delta >0$ sufficiently small (see, e.g., \cite{DM}).

Let $\phi _{1}$ be the positive eigenfunction defined by%
\begin{equation*}
-\Delta \phi _{1}+\phi _{1}=\lambda _{1}\phi _{1}\text{ \ in }\Omega ,\text{
\ }\phi _{1}=0\text{ \ on }\partial \Omega ,
\end{equation*}%
where $\lambda _{1}$ is the principle eigenvalue characterized by 
\begin{equation}
\lambda _{1}=\inf_{w\in \mathcal{H}_{0}^{1}(\Omega )\backslash \{0\}}\frac{%
\int_{\Omega }(|\nabla w|^{2}+|w|^{2})\,\mathrm{d}x}{\int_{\Omega }|w|^{2}\,%
\mathrm{d}x}.  \label{71}
\end{equation}

We will make use the topological degree theory to get a third solution for
system $(\mathrm{P})$. However, the singular terms in system $(\mathrm{P})$
prevent the degree calculation to be well defined. This is mainly due to the
difficulty in getting estimates from below for solutions of $(\mathrm{P})$.
To this challenge, we disturb system $(\mathrm{P})$ by introducing a
parameter $\varepsilon >0$. This gives rise to a regularized system for $(%
\mathrm{P})$ whose study is relevant to our initial problem.

For $\varepsilon \in (0,1),$ we state the regularized system%
\begin{equation*}
(\mathrm{P}_{\varepsilon })\qquad \left\{ 
\begin{array}{ll}
-\Delta u+u=sgn(v)\frac{(|u|+\varepsilon )^{\alpha _{1}}}{(|v|+\varepsilon
)^{\beta _{1}}} & \text{in }\Omega \\ 
-\Delta v+v=sgn(u)\frac{(|u|+\varepsilon )^{\alpha _{2}}}{(|v|+\varepsilon
)^{\beta _{2}}} & \text{in }\Omega \\ 
u,v=0 & \text{on }\partial \Omega .%
\end{array}%
\right.
\end{equation*}

The existence result regarding problem $(\mathrm{P}_{\varepsilon })$ is
stated as follows.

\begin{theorem}
\label{T4}Under assumption (\ref{alpha}), system $(\mathrm{P}_{\varepsilon
}) $ admits nontrivial solutions $(u_{\varepsilon },v_{\varepsilon })$ in $%
\mathcal{C}^{1}(\overline{\Omega })\times \mathcal{C}^{1}(\overline{\Omega }%
) $ satisfying%
\begin{equation}
-C^{-1}z_{1}(x)\leq u_{\varepsilon }(x)\leq C^{-1}z_{1}(x),\text{ }\forall
x\in \Omega ,  \label{14}
\end{equation}%
\begin{equation}
-C^{-1}z_{2}(x)\leq v_{\varepsilon }(x)\leq C^{-1}z_{2}(x),\text{ }\forall
x\in \Omega ,  \label{14*}
\end{equation}%
for all $\varepsilon \in (0,1)$ and all $C>1$. Moreover, there exists $%
(u_{\ast },v_{\ast })\in \mathcal{H}_{0}^{1}(\Omega )\times \mathcal{H}%
_{0}^{1}(\Omega )$, solution of problem $(\mathrm{P}),$ within%
\begin{equation*}
(u_{\ast },v_{\ast })\in \lbrack -C^{-1}z_{1},C^{-1}z_{1}]\times \lbrack
-C^{-1}z_{2},C^{-1}z_{2}],
\end{equation*}%
such that%
\begin{equation}
u_{\varepsilon }\rightarrow u_{\ast }\text{ and }v_{\varepsilon }\rightarrow
v_{\ast }\text{ in }\mathcal{H}_{0}^{1}(\Omega )\text{ as }\varepsilon
\rightarrow 0.  \label{15}
\end{equation}
\end{theorem}

\subsection{Topological degree results}

We shall study the homotopy class of problem%
\begin{equation*}
(\mathrm{P}_{\varepsilon ,\theta }^{t})\qquad \left\{ 
\begin{array}{l}
-\Delta u+u=\mathrm{F}_{1,\theta }^{t}({x,}u,v)\text{ in }\Omega , \\ 
-\Delta v+v=\mathrm{F}_{2,\theta }^{t}({x,}u,v)\text{ in }\Omega , \\ 
u,v=0\text{ \ on }\partial \Omega ,%
\end{array}%
\right.
\end{equation*}%
with%
\begin{equation}
\mathrm{F}_{1,\theta }^{t}({x,}u,v)=t\text{ }sgn(v)\frac{(|u|+\varepsilon
)^{\alpha _{1}}}{(|v|+\varepsilon )^{\beta _{1}}}+(1-t)sgn((1-\theta )v)%
\text{ }(1+\theta \lambda _{1}u^{+}),  \label{9}
\end{equation}%
\begin{equation}
\mathrm{F}_{2,\theta }^{t}({x,}u,v)=t\text{ }sgn(u)\frac{(|u|+\varepsilon
)^{\alpha _{2}}}{(|v|+\varepsilon )^{\beta _{2}}}+(1-t)sgn((1-\theta )u)%
\text{ }(1+\theta \lambda _{1}v^{+}),  \label{9*}
\end{equation}%
for $\varepsilon \in (0,1),$ for $t\in \lbrack 0,1],$ where $\theta $ is a
constant such that $\theta \in \{0,1\}$ and $s^{+}:=\max \{0,s\},$ for $s\in 
%TCIMACRO{\U{211d} }%
%BeginExpansion
\mathbb{R}
%EndExpansion
.$

It is worth noting that all solutions $(u,v)\in \mathcal{H}_{0}^{1}(\Omega
)\times \mathcal{H}_{0}^{1}(\Omega )$ of $(\mathrm{P}_{\varepsilon ,\theta
}^{t})$ satisfy 
\begin{equation}
u(x),v(x)\neq 0\text{ \ for a.e. }x\in \Omega ,  \label{10}
\end{equation}%
for all $t\in \lbrack 0,1],$ all $\varepsilon \in (0,1)$ and for $\theta \in
\{0,1\}$. This is due to the fact that $(0,0)$ cannot be a solution of $(%
\mathrm{P}_{\varepsilon ,\theta }^{t})$ because $\mathrm{F}_{i,\theta }^{t}({%
x,}0,0)\neq 0$ as well as the fact that "a.e. in $\Omega $" is an
equivalence relation in $L^{1}(\Omega )$.

\begin{remark}
\label{R2}The decoupled system $(\mathrm{P}_{\varepsilon ,1}^{0})$ (that is $%
(\mathrm{P}_{\varepsilon ,\theta }^{t})$ for $t=0$ and $\theta =1$) which
reads as 
\begin{equation*}
(\mathrm{P}_{\varepsilon ,1}^{0})\qquad \left\{ 
\begin{array}{l}
-\Delta u+u=\mathrm{F}_{1,1}^{0}({x,}u,v)=1+\lambda _{1}u^{+}\text{ in }%
\Omega \\ 
-\Delta v+v=\mathrm{F}_{2,1}^{0}({x,}u,v)=1+\lambda _{1}v^{+}\text{ in }%
\Omega \\ 
u,v=0\text{ \ on }\partial \Omega ,%
\end{array}%
\right.
\end{equation*}%
does not admit solutions $(u,v)\in \mathcal{H}_{0}^{1}(\Omega )\times 
\mathcal{H}_{0}^{1}(\Omega )$, for all $\varepsilon \in (0,1)$. This is due
to \cite[Proposition 9.64]{MMPA} with $p=2$ and $\beta (x),\xi
(x),h(x)\equiv 1$.
\end{remark}

In the sequel, we denote by $\mathcal{B}_{R_{\varepsilon }}$ and $\mathcal{B}%
_{z}$ the balls in $\mathcal{C}^{1}(\overline{\Omega })\times \mathcal{C}%
^{1}(\overline{\Omega }),$ centered at the origin, defined by%
\begin{equation*}
\mathcal{B}_{R_{\varepsilon }}:=\left\{ (u,v)\in \mathcal{C}^{1}(\overline{%
\Omega })\times \mathcal{C}^{1}(\overline{\Omega }):\Vert u\Vert _{C^{1}(%
\overline{\Omega })}+\Vert v\Vert _{C^{1}(\overline{\Omega }%
)}<R_{\varepsilon }\right\} ,
\end{equation*}%
\begin{equation*}
\mathcal{B}_{z}:=\left\{ (u,v)\in \mathcal{B}_{R_{\varepsilon
}}:-C^{-1}z_{1}\leq u\leq C^{-1}z_{1},\text{ \ }-C^{-1}z_{2}\leq v\leq
C^{-1}z_{2}\right\} ,
\end{equation*}%
where, without loss of generality, we assumed that $R_{\varepsilon
}>\max_{i=1,2}\left\Vert z_{i}\right\Vert _{\infty },$ for all $\varepsilon
\in (0,1).$ It is readily seen that $\mathcal{B}_{z}$ is an open set in $%
\mathcal{C}^{1}(\overline{\Omega })\times \mathcal{C}^{1}(\overline{\Omega }%
) $.

The next result shows that solutions of problem $(\mathrm{P}_{\varepsilon
,\theta }^{t})$ cannot occur outside the ball $\mathcal{B}_{R_{\varepsilon
}} $.

\begin{proposition}
\label{P2}Assume (\ref{alpha}) holds.\ Then, there is a constant $%
R_{\varepsilon }>0$ such that every solution $(u,v)$ of $(\mathrm{P}%
_{\varepsilon ,\theta }^{t})$ belongs to $\mathcal{C}^{1}(\overline{\Omega }%
)\times \mathcal{C}^{1}(\overline{\Omega })$ and satisfy 
\begin{equation}
\left\Vert u\right\Vert _{\mathcal{C}^{1}(\overline{\Omega })},\left\Vert
v\right\Vert _{\mathcal{C}^{1}(\overline{\Omega })}<R_{\varepsilon },
\label{31}
\end{equation}%
for all $t\in (0,1],$ all $\varepsilon \in (0,1)$ and for $\theta \in
\{0,1\} $. Moreover, if $\theta =0$, then all positive solutions $%
(u_{+},v_{+})$ and all negative solutions $(u_{-},v_{-})$ of $(\mathrm{P}%
_{\varepsilon ,0}^{t})$ ($(\mathrm{P}_{\varepsilon ,\theta }^{t})$ with $%
\theta =0$) satisfy 
\begin{equation}
\begin{array}{c}
C^{-1}z_{1}(x)\ll u_{+}(x)\text{ \ and \ }u_{-}(x)\ll -C^{-1}z_{1}(x)\text{,}%
\quad \forall x\in \Omega ,%
\end{array}
\label{11}
\end{equation}%
\begin{equation}
\begin{array}{c}
C^{-1}z_{2}(x)\ll v_{+}(x)\text{ \ and \ }v_{-}(x)\ll -C^{-1}z_{2}(x)\text{,}%
\quad \forall x\in \Omega ,%
\end{array}
\label{11*}
\end{equation}%
for a constant $C>1$ large.
\end{proposition}

\begin{proof}
We begin by proving (\ref{31}). If $\alpha _{i}\leq 0$ in (\ref{alpha}), by (%
\ref{9}) and (\ref{9*}) we have%
\begin{equation*}
\begin{array}{l}
|\mathrm{F}_{1,\theta }^{t}({x,}u,v)|\leq \varepsilon ^{\alpha _{1}+\beta
_{1}}+\lambda _{1}|u|\text{ \ and \ }|\mathrm{F}_{2,\theta }^{t}({x,}%
u,v)|\leq \varepsilon ^{\alpha _{2}+\beta _{2}}+\lambda _{1}|v|.%
\end{array}%
\end{equation*}%
Then, the regularity result \cite[Corollary 8.13]{MMPA} together with the
compact embedding $\mathcal{C}^{1,\tau }(\overline{\Omega })\subset \mathcal{%
C}^{1}(\overline{\Omega })$ show that all solutions $(u,v)$ of $(\mathrm{P}%
_{\varepsilon ,\theta }^{t})$ are bounded in $\mathcal{C}^{1}(\overline{%
\Omega })\times \mathcal{C}^{1}(\overline{\Omega })$ and satisfy (\ref{31}).

Let us examine the case when $\alpha _{i}>0$ in (\ref{alpha}), $i=1,2$. By
contradiction suppose that for every $n\in 
%TCIMACRO{\U{2115} }%
%BeginExpansion
\mathbb{N}
%EndExpansion
,$ there exist $t_{n}\in (0,1]$ and a solution $(u_{n},v_{n})$ of $(\mathrm{P%
}_{\varepsilon ,\theta }^{t_{n}})$ such that%
\begin{equation*}
t_{n}\rightarrow t\in (0,1]\text{ \ and \ }\Vert u_{n}\Vert _{\mathcal{C}%
^{1}(\overline{\Omega })},\Vert v_{n}\Vert _{\mathcal{C}^{1}(\overline{%
\Omega })}\rightarrow \infty \text{ \ as }n\rightarrow \infty .
\end{equation*}%
Without loss of generality we may admit that 
\begin{equation}
\begin{array}{c}
\gamma _{n}:=\Vert u_{n}\Vert _{\mathcal{C}^{1}(\overline{\Omega }%
)}\rightarrow \infty \text{ as }n\rightarrow \infty .%
\end{array}
\label{3}
\end{equation}%
Denote 
\begin{equation}
{\normalsize \tilde{u}}_{n}:=\frac{u_{n}}{\gamma _{n}}\in \mathcal{C}^{1}(%
\overline{\Omega })\text{ with }\Vert {\normalsize \tilde{u}}_{n}\Vert _{%
\mathcal{C}^{1}(\overline{\Omega })}=1,\text{ for all }n\in 
%TCIMACRO{\U{2115} }%
%BeginExpansion
\mathbb{N}
%EndExpansion
.  \label{25}
\end{equation}%
Problem $(\mathrm{P}_{\varepsilon ,\theta }^{t_{n}})$ results in 
\begin{equation}
\begin{array}{l}
-\Delta {\normalsize \tilde{u}}_{n}+{\normalsize \tilde{u}}_{n}=\frac{1}{%
\gamma _{n}}\mathrm{F}_{1,\theta }^{t_{n}}({x,}u_{n},v_{n}) \\ 
=\frac{t_{n}}{\gamma _{n}}\text{ }sgn(v_{n})\frac{(|u_{n}|+\varepsilon
)^{\alpha _{1}}}{(|v_{n}|+\varepsilon )^{\beta _{1}}}+\frac{1-t_{n}}{\gamma
_{n}}\text{ }sgn((1-\theta )v_{n})\text{ }(1+\theta \lambda _{1}u_{n}^{+}).%
\end{array}
\label{13}
\end{equation}%
By (\ref{3}) and since $1>\alpha _{1}>0$, one has%
\begin{equation*}
\begin{array}{l}
\left\vert \frac{t_{n}}{\gamma _{n}}\text{ }sgn(v_{n})\frac{%
(|u_{n}|+\varepsilon )^{\alpha _{1}}}{(|v_{n}|+\varepsilon )^{\beta _{1}}}+%
\frac{1-t_{n}}{\gamma _{n}}\text{ }sgn((1-\theta )v_{n})\text{ }(1+\theta
\lambda _{1}u_{n}^{+})\right\vert \\ 
\\ 
\leq \frac{1}{\gamma _{n}}\frac{(|u_{n}|+\varepsilon )^{\alpha _{1}}}{%
\varepsilon ^{\beta _{1}}}+1+\lambda _{1}\tilde{u}_{n}^{+}=\gamma
_{n}^{\alpha _{1}-1}\frac{(|\tilde{u}_{n}|+\frac{\varepsilon }{\gamma _{n}}%
)^{\alpha _{1}}}{\varepsilon ^{\beta _{1}}}+1+\lambda _{1}\tilde{u}_{n}^{+}
\\ 
\\ 
\leq \frac{(|\tilde{u}_{n}|+1)^{\alpha _{1}}}{\varepsilon ^{\beta _{1}}}%
+1+\lambda _{1}\tilde{u}_{n}^{+}\leq C_{\varepsilon }(1+\Vert \tilde{u}%
_{n}\Vert _{C^{1}(\overline{\Omega })}^{\alpha _{1}}+\Vert \tilde{u}%
_{n}\Vert _{C^{1}(\overline{\Omega })})\text{ in }\Omega ,%
\end{array}%
\end{equation*}%
for some constant $C_{\varepsilon }>0$ independent of $n$. Then, thanks to
the regularity up to the boundary in \cite{L}, we derive that $\tilde{u}_{n}$
is bounded in $\mathcal{C}^{1,\tau }(\overline{\Omega })$ for certain $\tau
\in (0,1)$. The compactness of the embedding $\mathcal{C}^{1,\tau }(%
\overline{\Omega })\subset \mathcal{C}^{1}(\overline{\Omega })$ implies 
\begin{equation*}
\tilde{u}_{n}\rightarrow \tilde{u}\text{ \ in }\mathcal{C}^{1}(\overline{%
\Omega })\text{.}
\end{equation*}%
Taking $\theta =0$ and passing to the limit in (\ref{13}) as $n\rightarrow
\infty ,$ we obtain%
\begin{equation*}
-\Delta {\normalsize \tilde{u}}+{\normalsize \tilde{u}}=0\text{ in }\Omega ,%
\text{ }{\normalsize \tilde{u}}=0\text{ on }\partial \Omega .
\end{equation*}%
Therefore ${\normalsize \tilde{u}=0}$ which contradicts (\ref{25}). If $%
\theta =1,$ in the limit results in%
\begin{equation}
\left\{ 
\begin{array}{ll}
-\Delta {\normalsize \tilde{u}}+{\normalsize \tilde{u}}=(1-t)\lambda _{1}%
{\normalsize \tilde{u}}^{+} & \text{ in }\Omega , \\ 
{\normalsize \tilde{u}}=0, & \text{on }\partial \Omega%
\end{array}%
\right. ,\text{ for }t\in (0,1].  \label{19}
\end{equation}%
Test with $-{\normalsize \tilde{u}}$ in (\ref{19}) leads to ${\normalsize 
\tilde{u}}={\normalsize \tilde{u}^{+}}$, which is nonzero because of (\ref%
{25}). Acting with ${\normalsize \tilde{u}}$ in (\ref{19}) and integrating
over $\Omega ,$ we get 
\begin{equation*}
\int_{\Omega }(|\nabla {\normalsize \tilde{u}}|^{2}+{\normalsize \tilde{u}}%
^{2})\,\mathrm{d}x=(1-t)\lambda _{1}\int_{\Omega }{\normalsize \tilde{u}^{2}}%
\,\mathrm{d}x,\text{ for }t\in (0,1].
\end{equation*}%
Put $t=1$ implies ${\normalsize \tilde{u}}=0$ which contradicts the fact
that ${\normalsize \tilde{u}}\neq 0.$ If $t\in (0,1)$, (\ref{71}) implies 
\begin{equation*}
\lambda _{1}\int_{\Omega }{\normalsize \tilde{u}^{2}}\,\mathrm{d}x\leq
(1-t)\lambda _{1}\int_{\Omega }{\normalsize \tilde{u}^{2}}\,\mathrm{d}x,
\end{equation*}%
which is absurd because ${\normalsize \tilde{u}}>0$ and $t<1.$ Consequently,
this shows that there exists a constant $R_{\varepsilon }>0$ such that (\ref%
{31}) holds true.

We proceed to show (\ref{11}) and (\ref{11*}). Let $(u_{+},v_{+})$ be a
positive solution of $(\mathrm{P}_{\varepsilon ,0}^{t})$. Then $%
sgn(u_{+}),sgn(v_{+})\equiv 1.$ On account of (\ref{alpha}), (\ref{21}) and (%
\ref{31}), we have%
\begin{equation*}
\begin{array}{l}
\left\{ 
\begin{array}{ll}
C^{-1}d(x)^{\alpha _{i}-\beta _{i}} & \text{in\ }\Omega \backslash \overline{%
\Omega }_{\delta } \\ 
-C^{-1} & \text{in }\Omega _{\delta }%
\end{array}%
\right. \leq \left\{ 
\begin{array}{ll}
C^{-1}\delta ^{\alpha _{i}-\beta _{i}} & \text{in\ }\Omega \backslash 
\overline{\Omega }_{\delta } \\ 
-C^{-1} & \text{in }\Omega _{\delta }%
\end{array}%
\right. \\ 
\\ 
<\left\{ 
\begin{array}{ll}
t\frac{\varepsilon ^{\alpha _{i}}}{(R_{\varepsilon }+1)^{\beta _{i}}}+(1-t)
& \text{if }\alpha _{i}\geq 0 \\ 
t(R_{\varepsilon }+1)^{\alpha _{i}-\beta _{i}}+(1-t) & \text{if }\alpha
_{i}\leq 0%
\end{array}%
\right. \\ 
\\ 
\leq \left\{ 
\begin{array}{ll}
t\frac{\varepsilon ^{\alpha _{i}}}{(|v_{+}|+1)^{\beta _{i}}}+(1-t) & \text{%
if }\alpha _{i}\geq 0 \\ 
t\frac{(|u_{+}|+1)^{\alpha _{i}}}{(|v_{+}|+1)^{\beta _{i}}}+(1-t) & \text{if 
}\alpha _{i}\leq 0%
\end{array}%
\right. \leq \mathrm{F}_{i,0}^{t}({x,}u_{+},v_{+})\text{ \ in\ }\Omega ,%
\text{ }i=1,2,%
\end{array}%
\end{equation*}%
for all $t\in \lbrack 0,1]$ and all $\varepsilon \in (0,1),$ provided $C>1$
is large. Thus, for each compact set $\mathrm{K}\subset \subset \Omega $,
there is a constant $\sigma =\sigma (\mathrm{K})>0$ such that%
\begin{equation*}
\sigma +C^{-1}\left\{ 
\begin{array}{ll}
d(x)^{\alpha _{i}-\beta _{i}} & \text{in \ }\Omega \backslash \overline{%
\Omega }_{\delta }, \\ 
-1 & \text{in \ }\Omega _{\delta },%
\end{array}%
\right. <\mathrm{F}_{i,0}^{t}({x,}u_{+},v_{+})\text{ a.e. in }\Omega \cap 
\mathrm{K},i=1,2.
\end{equation*}%
Then, by (\ref{22}) and the strong comparison principle \cite[Proposition $%
2.6$]{AR}, we infer that $C^{-1}z_{1}(x)\ll u_{+}(x)$ and $C^{-1}z_{2}(x)\ll
v_{+}(x)$, for a.a. $x\in \overline{\Omega }$. In the same manner we can
show that $-C^{-1}z_{1}(x)\gg u_{-}(x)$ and $-C^{-1}z_{2}(x)\gg v_{-}(x)$,
for a.a. $x\in \overline{\Omega }$. This ends the proof.
\end{proof}

On account of (\ref{10}), $\mathrm{F}_{i,\theta }^{t}{(x},\cdot ,\cdot )$ is
continuous for a.e. $x\in \Omega ,$ for $i=1,2$. Thus, the homotopy $%
\mathcal{H}_{\varepsilon ,\theta }:[0,1]\times \mathcal{C}^{1}(\overline{%
\Omega })\times \mathcal{C}^{1}(\overline{\Omega })\rightarrow \mathcal{C}(%
\overline{\Omega })$ given by%
\begin{equation*}
\mathcal{H}_{\varepsilon ,\theta }(t,u,v)=I(u,v)-\left( 
\begin{array}{cc}
(-\Delta +I)^{-1} & 0 \\ 
0 & (-\Delta +I)^{-1}%
\end{array}%
\right) \left( 
\begin{array}{l}
\mathrm{F}_{1,\theta }^{t}({x,}u,v) \\ 
\multicolumn{1}{c}{\mathrm{F}_{2,\theta }^{t}({x,}u,v)}%
\end{array}%
\right)
\end{equation*}%
is well defined for all $t\in \lbrack 0,1],$ all $\varepsilon \in (0,1)$,
and for $\theta =0,1$. Moreover, the compactness of the operator $(-\Delta
+I)^{-1}:\mathcal{C}(\overline{\Omega })\rightarrow \mathcal{C}^{1}(%
\overline{\Omega })$ implies that $\mathcal{H}_{\varepsilon ,\theta }$ is
completely continuous.

The next result provides the values of the the topological degree of $%
\mathcal{H}_{\varepsilon ,\theta }$ in certain specific sets for $\theta =0$
and $\theta =1$.

\begin{proposition}
\label{P5}Assume that (\ref{alpha}) is satisfied. Then, the Leray-Schauder
topological degrees $\deg (\mathcal{H}_{\varepsilon ,1}(1,\cdot ,\cdot ),%
\mathcal{B}_{R_{\varepsilon }},0)$ (with $\theta =1$) and $\deg (\mathcal{H}%
_{\varepsilon ,0}(1,\cdot ,\cdot ),\mathcal{B}_{R_{\varepsilon }}\backslash 
\overline{\mathcal{B}}_{z},0)$ (with $\theta =0$) are well defined for all $%
t\in \lbrack 0,1]$ and all $\varepsilon \in (0,1)$. Moreover, it holds%
\begin{equation}
\deg (\mathcal{H}_{\varepsilon ,1}(1,\cdot ,\cdot ),\mathcal{B}%
_{R_{\varepsilon }},0)=\deg (\mathcal{H}_{\varepsilon ,1}(0,\cdot ,\cdot ),%
\mathcal{B}_{R_{\varepsilon }},0)=0  \label{12}
\end{equation}%
and%
\begin{equation}
\deg (\mathcal{H}_{\varepsilon ,0}(1,\cdot ,\cdot ),\mathcal{B}%
_{R_{\varepsilon }}\backslash \overline{\mathcal{B}}_{z},0)=\deg (\mathcal{H}%
_{\varepsilon ,0}(0,\cdot ,\cdot ),\mathcal{B}_{R_{\varepsilon }}\backslash 
\overline{\mathcal{B}}_{z},0)\neq 0,  \label{12*}
\end{equation}%
where $\overline{\mathcal{B}}_{z}$ is the closure of $\mathcal{B}_{z}$ in $%
\mathcal{C}^{1}(\overline{\Omega })\times \mathcal{C}^{1}(\overline{\Omega }%
) $.
\end{proposition}

\begin{proof}
Proposition \ref{P2} expressly establishes that solutions of $(\mathrm{P}%
_{\varepsilon ,\theta }^{t})$ must lie in $\mathcal{B}_{R_{\varepsilon }}$
and, if $\theta =0$, positive and negative solutions of $(\mathrm{P}%
_{\varepsilon ,0}^{t})$ are located in $\mathcal{B}_{R_{\varepsilon
}}\backslash \overline{\mathcal{B}}_{z}$. Hence, the degrees $\deg (\mathcal{%
H}_{\varepsilon ,1}(1,\cdot ,\cdot ),\mathcal{B}_{R_{\varepsilon }},0)$ and $%
\deg (\mathcal{H}_{\varepsilon ,0}(1,\cdot ,\cdot ),\mathcal{B}%
_{R_{\varepsilon }}\backslash \overline{\mathcal{B}}_{z},0)$ are well
defined for all $\varepsilon \in (0,1)$. Moreover, the homotopy invariance
property of the degree ensures that the first equality in (\ref{12}) and (%
\ref{12*}) is fulfilled. On the other hand, according to Remark \ref{R2},
problem $(\mathrm{P}_{\varepsilon ,1}^{0})$ ($(\mathrm{P}_{\varepsilon
,\theta }^{t})$ with $t=0$ and $\theta =1$) has no solutions whereas $(%
\mathrm{P}_{\varepsilon ,0}^{1})$ ($(\mathrm{P}_{\varepsilon ,\theta }^{t})$
with $t=1$ and $\theta =0$), that is reduced to the following decoupled
torsion problems%
\begin{equation*}
\left\{ 
\begin{array}{l}
-\Delta u+u=1\text{ in }\Omega \\ 
-\Delta v+v=1\text{ in }\Omega%
\end{array}%
\right. ,\text{ }u,v=0\text{ on }\partial \Omega ,
\end{equation*}%
admits a unique solution. Thence 
\begin{equation*}
\deg \left( \mathcal{H}_{\varepsilon ,1}(0,\cdot ,\cdot ),\mathcal{B}%
_{R_{\varepsilon }},0\right) =0,
\end{equation*}%
while%
\begin{equation*}
\deg (\mathcal{H}_{\varepsilon ,0}(0,\cdot ,\cdot ),\mathcal{B}%
_{R_{\varepsilon }}\backslash \overline{\mathcal{B}}_{z},0)\neq 0,
\end{equation*}%
for all $\varepsilon \in (0,1).$ Consequently, we deduce that 
\begin{equation*}
\deg \left( \mathcal{H}_{\varepsilon ,1}(1,\cdot ,\cdot ),\mathcal{B}%
_{R_{\varepsilon }},0\right) =0\text{ \ and \ }\deg (\mathcal{H}%
_{\varepsilon ,0}(1,\cdot ,\cdot ),\mathcal{B}_{R_{\varepsilon }}\backslash 
\overline{\mathcal{B}}_{z},0)\neq 0,
\end{equation*}%
for all $\varepsilon \in (0,1).$ This completes the proof.
\end{proof}

\subsection{\textbf{Proof of Theorem \protect\ref{T4}.}}

By the definition of the homotopy $\mathcal{H}_{\varepsilon ,\theta }$,
observe that%
\begin{equation*}
\mathcal{H}_{\varepsilon ,0}(1,\cdot ,\cdot )=\mathcal{H}_{\varepsilon
,1}(1,\cdot ,\cdot ),\text{ for all }\varepsilon \in (0,1).
\end{equation*}%
Moreover, $(u,v)$ is a solution for $(\mathrm{P}_{\varepsilon })$ if, and
only if, 
\begin{equation*}
\begin{array}{c}
(u,v)\in \mathcal{B}_{R_{\varepsilon }}(0)\,\,\,\mbox{and}\,\,\,\mathcal{H}%
_{\varepsilon ,1}(1,u,v)=0.%
\end{array}%
\end{equation*}%
By virtue of the domain additivity property of Leray-Schauder degree we have 
\begin{eqnarray*}
&&\deg (\mathcal{H}_{\varepsilon ,1}(1,\cdot ,\cdot ),\mathcal{B}%
_{R_{\varepsilon }},0) \\
&=&\deg (\mathcal{H}_{\varepsilon ,1}(1,\cdot ,\cdot ),\mathcal{B}%
_{R_{\varepsilon }}\backslash \overline{\mathcal{B}}_{z},0)+\deg (\mathcal{H}%
_{\varepsilon ,1}(1,\cdot ,\cdot ),\mathcal{B}_{z},0).
\end{eqnarray*}%
Then, on the basis of Proposition \ref{P5}, we infer that%
\begin{equation*}
\deg (\mathcal{H}_{\varepsilon ,1}(1,\cdot ,\cdot ),\mathcal{B}_{z},0)\neq 0,
\end{equation*}%
showing that problem $(\mathrm{P}_{\varepsilon })$ has a solution $%
(u_{\varepsilon },v_{\varepsilon })$ in $\mathcal{B}_{z}$. Consequently, $%
(u_{\varepsilon },v_{\varepsilon })\in \mathcal{C}^{1}(\overline{\Omega }%
)\times \mathcal{C}^{1}(\overline{\Omega })$ fulfills (\ref{14})-(\ref{14*})
while (\ref{10}) forces $u_{\varepsilon },v_{\varepsilon }\neq 0$ \ a.e in $%
\Omega ,$ for all $\varepsilon \in (0,1)$. This proves the first part of the
theorem.

We proceed to show the limit in (\ref{15}). Set $\varepsilon =\frac{1}{n}$
in $(\mathrm{P}_{\varepsilon })$ with any positive integer $n\geq 1.$ From
above, there exists $(u_{n},v_{n}):=(u_{\frac{1}{n}},v_{\frac{1}{n}})\in 
\mathcal{C}^{1}(\overline{\Omega })\times \mathcal{C}^{1}(\overline{\Omega }%
) $ solution of $(\mathrm{P}_{n})$ ($(\mathrm{P}_{\varepsilon })$ with $%
\varepsilon =\frac{1}{n}$) such that 
\begin{equation}
(u_{n},v_{n})\in \lbrack -C^{-1}z_{1},C^{-1}z_{1}]\times \lbrack
-C^{-1}z_{2},C^{-1}z_{2}]  \label{26}
\end{equation}%
and 
\begin{equation}
\left\{ 
\begin{array}{l}
\int_{\Omega }(\nabla u_{n}\nabla \varphi _{1}+u_{n}\text{\thinspace }%
\varphi _{1})\text{ }\mathrm{d}x=\int_{\Omega }sgn(v_{n})\frac{(|u_{n}|+%
\frac{1}{n})^{\alpha _{1}}}{(|v_{n}|+\frac{1}{n})^{\beta _{1}}}\varphi _{1}\ 
\mathrm{d}x, \\ 
\int_{\Omega }(\nabla v_{n}\nabla \varphi _{2}+v_{n}\text{\thinspace }%
\varphi _{2})\text{ }\mathrm{d}x=\int_{\Omega }sgn(u_{n})\frac{(|u_{n}|+%
\frac{1}{n})^{\alpha _{2}}}{(|v_{n}|+\frac{1}{n})^{\beta _{2}}}\varphi _{2}\ 
\mathrm{d}x,%
\end{array}%
\right.  \label{122}
\end{equation}%
for all $\varphi _{1},\varphi _{2}\in \mathcal{H}_{0}^{1}(\Omega )$.

Taking $t=1$ and $\varepsilon =\frac{1}{n}$ in (\ref{9}), by (\ref{10}), we
infer that%
\begin{equation}
u_{n}(x),v_{n}(x)\neq 0\text{ for a.e. }x\in \Omega ,\text{ for all }n.
\label{27}
\end{equation}%
Acting with $(\varphi _{1},\varphi _{2})=(u_{n},v_{n})$ in (\ref{122}), by (%
\ref{alpha}), (\ref{26}) and (\ref{27}), we get 
\begin{equation*}
\begin{array}{l}
\int_{\Omega }(|\nabla u_{n}|^{2}+|u_{n}|^{2})\text{ }\mathrm{d}%
x=\int_{\Omega }sgn(v_{n})\frac{(|u_{n}|+\frac{1}{n})^{\alpha _{1}}}{%
(|v_{n}|+\frac{1}{n})^{\beta _{1}}}u_{n}\ \mathrm{d}x \\ 
\leq \int_{\Omega }\left\vert sgn(v_{n})\frac{(|u_{n}|+\frac{1}{n})^{\alpha
_{1}}}{(|v_{n}|+\frac{1}{n})^{\beta _{1}}}u_{n}\right\vert \ \mathrm{d}x \\ 
\leq \left\{ 
\begin{array}{ll}
\int_{\Omega }\frac{|u_{n}|^{\alpha _{1}+1}}{|v_{n}|^{\beta _{1}}}\ \mathrm{d%
}x & \text{if }\alpha _{1}\leq 0 \\ 
\int_{\Omega }\frac{(|u_{n}|+1)^{\alpha _{1}+1}}{|v_{n}|^{\beta _{1}}}\ 
\mathrm{d}x & \text{if }\alpha _{1}>0%
\end{array}%
\right. \\ 
\leq \left\{ 
\begin{array}{ll}
\int_{\Omega }\frac{(C^{-1}z_{1})^{\alpha _{1}+1}}{|v_{n}|^{\beta _{1}}}\ 
\mathrm{d}x & \text{if }\alpha _{1}\leq 0 \\ 
\int_{\Omega }\frac{(C^{-1}z_{1}+1)^{\alpha _{1}+1}}{|v_{n}|^{\beta _{1}}}\ 
\mathrm{d}x & \text{if }\alpha _{1}>0%
\end{array}%
\right. <\infty \ 
\end{array}%
\end{equation*}%
and%
\begin{equation*}
\begin{array}{l}
\int_{\Omega }(|\nabla v_{n}|^{2}+|v_{n}|^{2})\text{ }\mathrm{d}%
x=\int_{\Omega }sgn(u_{n})\frac{(|u_{n}|+\frac{1}{n})^{\alpha _{2}}}{%
(|v_{n}|+\frac{1}{n})^{\beta _{2}}}v_{n}\ \mathrm{d}x \\ 
\leq \int_{\Omega }\left\vert sgn(u_{n})\frac{(|u_{n}|+\frac{1}{n})^{\alpha
_{2}}}{(|v_{n}|+\frac{1}{n})^{\beta _{2}}}v_{n}\right\vert \ \mathrm{d}x \\ 
\leq \left\{ 
\begin{array}{ll}
\int_{\Omega }|u_{n}|^{\alpha _{1}}|v_{n}|^{1-\beta _{1}}\ \mathrm{d}x & 
\text{if }\alpha _{1}\leq 0 \\ 
\int_{\Omega }(|u_{n}|+1)^{\alpha _{1}}|v_{n}|^{1-\beta _{1}}\ \mathrm{d}x & 
\text{if }\alpha _{1}>0%
\end{array}%
\right. \\ 
\leq \left\{ 
\begin{array}{ll}
\int_{\Omega }|u_{n}|^{\alpha _{1}}|v_{n}|^{1-\beta _{1}}\ \mathrm{d}x & 
\text{if }\alpha _{1}\leq 0 \\ 
\int_{\Omega }(C^{-1}z_{1}+1)^{\alpha _{1}}|v_{n}|^{1-\beta _{1}}\ \mathrm{d}%
x & \text{if }\alpha _{1}>0%
\end{array}%
\right. <\infty ,%
\end{array}%
\end{equation*}%
showing that $\{u_{n}\}_{n}$ and $\{v_{n}\}_{n}$ are bounded in $\mathcal{H}%
_{0}^{1}(\Omega )$. We are thus allowed to extract a subsequence (still
denoted by $\{u_{n}\}_{n},\{v_{n}\}_{n}$) such that 
\begin{equation}
\begin{array}{c}
u_{n}\rightharpoonup u_{\ast }\text{ \ and \ }v_{n}\rightharpoonup v_{\ast }%
\text{ \ in }\mathcal{H}_{0}^{1}(\Omega ).%
\end{array}
\label{130}
\end{equation}%
Moreover, on account of (\ref{26}) and (\ref{130}), we have 
\begin{equation*}
-C^{-1}z_{1}\leq u_{\ast }\leq C^{-1}z_{1}\ \text{ and \ }-C^{-1}z_{2}\leq
v_{\ast }\leq C^{-1}z_{2}\text{ \ in }\Omega .
\end{equation*}%
Inserting $(\varphi _{1},\varphi _{2})=(u_{n}-u_{\ast },v_{n}-v_{\ast })$ in
(\ref{122}) yields%
\begin{equation*}
\int_{\Omega }(\nabla u_{n}\text{\thinspace }\nabla (u_{n}-u_{\ast
})+u_{n}(u_{n}-u_{\ast }))\text{ }\mathrm{d}x=\int_{\Omega }sgn(v_{n})\frac{%
(|u_{n}|+\frac{1}{n})^{\alpha _{1}}}{(|v_{n}|+\frac{1}{n})^{\beta _{1}}}%
(u_{n}-u_{\ast })\ \mathrm{d}x
\end{equation*}%
\begin{equation*}
\int_{\Omega }(\nabla v_{n}\text{\thinspace }\nabla (v_{n}-v_{\ast
})+v_{n}(v_{n}-v_{\ast }))\text{ }\mathrm{d}x=\int_{\Omega }sgn(u_{n})\frac{%
(|u_{n}|+\frac{1}{n})^{\alpha _{2}}}{(|v_{n}|+\frac{1}{n})^{\beta _{2}}}%
(v_{n}-v_{\ast })\ \mathrm{d}x
\end{equation*}%
By (\ref{alpha}), (\ref{26}) and for $C>1$, we have 
\begin{equation*}
\begin{array}{l}
\left\vert sgn(v_{n})\frac{(|u_{n}|+\frac{1}{n})^{\alpha _{1}}}{(|v_{n}|+%
\frac{1}{n})^{\beta _{1}}}(u_{n}-u_{\ast })\right\vert \leq \left\{ 
\begin{array}{ll}
2C^{-1}z_{1}\frac{|u_{n}|^{\alpha _{1}}}{|v_{n}|^{\beta _{1}}} & \text{if }%
\alpha _{1}<0 \\ 
2C^{-1}z_{1}\frac{(C^{-1}z_{1}+1)^{\alpha _{1}}}{|v_{n}|^{\beta _{1}}} & 
\text{if }\alpha _{1}\geq 0%
\end{array}%
\right. \\ 
\leq \left\{ 
\begin{array}{ll}
2\left\Vert z_{1}\right\Vert _{\infty }\frac{|u_{n}|^{\alpha _{1}}}{%
|v_{n}|^{\beta _{1}}} & \text{if }\alpha _{1}<0 \\ 
2\left\Vert z_{1}\right\Vert _{\infty }\frac{(\left\Vert z_{1}\right\Vert
_{\infty }+1)^{\alpha _{1}}}{|v_{n}|^{\beta _{1}}} & \text{if }\alpha
_{1}\geq 0,%
\end{array}%
\right.%
\end{array}%
\end{equation*}%
while, by (\ref{27}), we infer that%
\begin{equation}
sgn(v_{n})\frac{(|u_{n}|+\frac{1}{n})^{\alpha _{1}}}{(|v_{n}|+\frac{1}{n}%
)^{\beta _{1}}}(u_{n}-u_{\ast })\in L^{1}(\Omega ).  \label{28}
\end{equation}%
Then, using (\ref{130}), (\ref{28}) and applying Fatou's Lemma, we derive
that%
\begin{equation*}
\underset{n\rightarrow \infty }{\lim }\int_{\Omega }(\nabla u_{n}\text{%
\thinspace }\nabla (u_{n}-u_{\ast })+u_{n}(u_{n}-u_{\ast }))\text{ }\mathrm{d%
}x\leq 0.
\end{equation*}%
Therefore, the $\mathcal{S}_{+}$-property of $-\Delta +I$ on $\mathcal{H}%
_{0}^{1}(\Omega )$ (see, e.g., \cite[Proposition 3.5]{MMPA}) guarantees that 
\begin{equation}
\begin{array}{c}
u_{n}\rightarrow u_{\ast }\text{ in }\mathcal{H}_{0}^{1}(\Omega ).%
\end{array}
\label{13*}
\end{equation}%
In the same manner, we show that%
\begin{equation}
\begin{array}{c}
v_{n}\rightarrow v_{\ast }\text{ in }\mathcal{H}_{0}^{1}(\Omega ).%
\end{array}
\label{13**}
\end{equation}%
On the other hand, by (\ref{alpha}), (\ref{13*}), (\ref{26}) and (\ref{13**}%
), it holds 
\begin{eqnarray*}
|sgn(v_{n})\frac{(|u_{n}|+\frac{1}{n})^{\alpha _{1}}}{(|v_{n}|+\frac{1}{n}%
)^{\beta _{1}}}\varphi _{1}| &\leq &\left\{ 
\begin{array}{ll}
\frac{|u_{n}|^{\alpha _{1}}}{|v_{n}|^{\beta _{1}}}|\varphi _{1}| & \text{if }%
\alpha _{1}<0 \\ 
\frac{(|u_{n}|+1)^{\alpha _{1}}}{|v_{n}|^{\beta _{1}}}|\varphi _{1}| & \text{%
if }\alpha _{1}\geq 0%
\end{array}%
\right. \\
&\leq &\left\{ 
\begin{array}{ll}
\frac{|u_{n}|^{\alpha _{1}}}{|v_{n}|^{\beta _{1}}}|\varphi _{1}| & \text{if }%
\alpha _{1}<0 \\ 
\frac{(\left\Vert z_{1}\right\Vert _{\infty }+1)^{\alpha _{1}}}{%
|v_{n}|^{\beta _{1}}}|\varphi _{1}| & \text{if }\alpha _{1}\geq 0%
\end{array}%
\right.
\end{eqnarray*}%
and%
\begin{eqnarray*}
|sgn(u_{n})\frac{(|u_{n}|+\frac{1}{n})^{\alpha _{2}}}{(|v_{n}|+\frac{1}{n}%
)^{\beta _{2}}}\varphi _{2}| &\leq &\left\{ 
\begin{array}{ll}
\frac{|u_{n}|^{\alpha _{2}}}{|v_{n}|^{\beta _{2}}}|\varphi _{2}| & \text{if }%
\alpha _{2}<0 \\ 
\frac{(|u_{n}|+1)^{\alpha _{2}}}{|v_{n}|^{\beta _{2}}}|\varphi _{2}| & \text{%
if }\alpha _{2}\geq 0%
\end{array}%
\right. \\
&\leq &\left\{ 
\begin{array}{ll}
\frac{|u_{n}|^{\alpha _{2}}}{|v_{n}|^{\beta _{2}}}|\varphi _{2}| & \text{if }%
\alpha _{2}<0 \\ 
\frac{(\left\Vert z_{1}\right\Vert _{\infty }+1)^{\alpha _{2}}}{%
|v_{n}|^{\beta _{2}}}|\varphi _{2}| & \text{if }\alpha _{2}\geq 0%
\end{array}%
\right. ,
\end{eqnarray*}%
for all $\varphi _{1},\varphi _{2}\in \mathcal{H}_{0}^{1}(\Omega ).$ Then,
on the basis of (\ref{27}), Lebesgue's dominated convergence theorem entails 
\begin{eqnarray*}
\lim_{n\rightarrow \infty }\int_{\Omega }sgn(v_{n})\frac{(|u_{n}|+\frac{1}{n}%
)^{\alpha _{1}}}{(|v_{n}|+\frac{1}{n})^{\beta _{1}}}\varphi _{1}\ \mathrm{d}%
x &=&\int_{\Omega }sgn(v_{\ast })\frac{|u_{\ast }|^{\alpha _{1}}}{|v_{\ast
}|^{\beta _{1}}}\varphi _{1}\ \mathrm{d}x, \\
\lim_{n\rightarrow \infty }\int_{\Omega }sgn(u_{n})\frac{(|u_{n}|+\frac{1}{n}%
)^{\alpha _{2}}}{(|v_{n}|+\frac{1}{n})^{\beta _{2}}}\varphi _{2}\ \mathrm{d}%
x &=&\int_{\Omega }sgn(u_{\ast })\frac{|u_{\ast }|^{\alpha _{2}}}{|v_{\ast
}|^{\beta _{2}}}\varphi _{2}\ \mathrm{d}x,
\end{eqnarray*}%
for all $\varphi _{1},\varphi _{2}\in \mathcal{H}_{0}^{1}(\Omega )$\textbf{. 
}Hence, we may pass to the limit in (\ref{122}) to conclude that $(u_{\ast
},v_{\ast })$ is a solution of problem $(\mathrm{P})$ within $%
[-C^{-1}z_{1},C^{-1}z_{1}]\times \lbrack -C^{-1}z_{2},C^{-1}z_{2}]$.

\section{\textbf{Proof of }Theorem \protect\ref{T2}}

\label{S3}

This section is devoted to the proof of the main result Theorem \ref{T2}. It
is performed into two steps, distinguishing the study of constant-sign
solutions from that of nodal solutions.

\begin{remark}
\label{R1}All solutions $(u,v)\in \mathcal{H}_{0}^{1}(\Omega )\times 
\mathcal{H}_{0}^{1}(\Omega )$ of $(\mathrm{P})$ satisfy $u(x),v(x)\neq 0$
for a.e. $x\in \Omega $. This is due to the singular character at the origin
of right hand-side of the equations in $(\mathrm{P})$ together with the fact
that "a.e. in $\Omega $" is an equivalence relation in $L^{1}(\Omega )$.
\end{remark}

\subsection{Constant sign solutions}

\begin{proposition}
\label{T5}Assume (\ref{alpha}) is fulfilled with $\alpha _{i}\leq 0,$ $i=1,2$%
. Then, problem $(\mathrm{P})$ does not admit more than two opposite
constant-sign solutions $(u_{+},v_{+})$ and $(u_{-},v_{-})$ in $(\mathcal{H}%
_{0}^{1}(\Omega )\cap L^{\infty }(\Omega ))^{2}$.
\end{proposition}

\begin{proof}
We only show the existence of a unique positive solution $(u_{+},v_{+})$ for
problem $(\mathrm{P})$ because the existence of the negative solution $%
(u_{1,-},u_{2,-})$ can be justified similarly. By contradiction, let $%
(u_{+},v_{+})$ and $(\hat{u}_{+},\hat{v}_{+})$ be two distinct positive
solutions of $(\mathrm{P})$ in $(\mathcal{H}_{0}^{1}(\Omega )\cap L^{\infty
}(\Omega ))^{2}$ satisfying%
\begin{equation*}
\left\{ 
\begin{array}{ll}
-\Delta u_{+}+u_{+}=\frac{|u_{+}|^{\alpha _{1}}}{|v_{+}|^{\beta _{1}}} & 
\text{in }\Omega \\ 
-\Delta v_{+}+v_{+}=\frac{|u_{+}|^{\alpha _{2}}}{|v_{+}|^{\beta _{2}}} & 
\text{in }\Omega \\ 
u_{+},v_{+}=0 & \text{on }\partial \Omega ,%
\end{array}%
\right.
\end{equation*}%
\begin{equation*}
\left\{ 
\begin{array}{ll}
-\Delta \hat{u}_{+}+\hat{u}_{+}=\frac{|\hat{u}_{+}|^{\alpha _{1}}}{|\hat{v}%
_{+}|^{\beta _{1}}} & \text{in }\Omega \\ 
-\Delta \hat{v}_{+}+\hat{v}_{+}=\frac{|\hat{u}_{+}|^{\alpha _{2}}}{|\hat{v}%
_{+}|^{\beta _{2}}} & \text{in }\Omega \\ 
\hat{u}_{+},\hat{v}_{+}=0 & \text{on }\partial \Omega ,%
\end{array}%
\right.
\end{equation*}%
where $sgn(u_{+}),sgn(v_{+}),sgn(\hat{u}_{+}),sgn(\hat{u}_{+})\equiv 1$. Thus%
\begin{equation}
\left\{ 
\begin{array}{l}
-\Delta u_{+}+\Delta \hat{u}_{+}+u_{+}-\hat{u}_{+}=\frac{|u_{+}|^{\alpha
_{1}}}{|v_{+}|^{\beta _{1}}}-\frac{|\hat{u}_{+}|^{\alpha _{1}}}{|\hat{v}%
_{+}|^{\beta _{1}}}\text{ in }\Omega \\ 
u_{+}-\hat{u}_{+}=0\text{ \ on }\partial \Omega .%
\end{array}%
\right.  \label{17*}
\end{equation}%
Multiply (\ref{17*}) by $(u_{+}-\hat{u}_{+})$ and integrate over $\Omega ,$
the assumption $\alpha _{i}\leq 0$ yields%
\begin{equation*}
\begin{array}{l}
0\leq \int_{\Omega }|\nabla (u_{+}-\hat{u}_{+})|^{2}\ dx+\int_{\Omega
}|u_{+}-\hat{u}_{+}|^{2}\ dx \\ 
=\int_{\Omega }\left( \frac{|u_{+}|^{\alpha _{1}}}{|v_{+}|^{\beta _{1}}}-%
\frac{|\hat{u}_{+}|^{\alpha _{1}}}{|\hat{v}_{+}|^{\beta _{1}}}\right) (u_{+}-%
\hat{u}_{+})\ dx\leq 0,%
\end{array}%
\end{equation*}%
showing that $u_{+}=\hat{u}_{+}$ in $\Omega $. A similar argument produces $%
v_{+}=\hat{v}_{+}.$
\end{proof}

\begin{theorem}
\label{T1}Under assumption (\ref{alpha}), problem $(\mathrm{P})$ admits at
least two opposite constant-sign solutions $(u_{+},v_{+})$ and $%
(u_{-},v_{-}) $ in $\mathcal{C}^{1,\tau }(\overline{\Omega })\times \mathcal{%
C}^{1,\tau }(\overline{\Omega }),$ for certain $\tau \in (0,1).$ Moreover,
for a constant $C>0$ large, it holds%
\begin{equation}
C^{-1}z_{1}(x)\ll u_{+}(x)\ll Cy_{1}(x),\text{ }\forall x\in \Omega ,
\label{p1}
\end{equation}%
\begin{equation}
C^{-1}z_{2}(x)\ll v_{+}(x)\ll Cy_{2}(x),\text{ }\forall x\in \Omega ,
\end{equation}%
\begin{equation}
-Cy_{1}(x)\ll u_{-}(x)\ll -C^{-1}z_{1}(x),\text{ }\forall x\in \Omega
\end{equation}%
and%
\begin{equation}
-Cy_{2}(x)\ll v_{-}(x)\ll -C^{-1}z_{2}(x),\text{ }\forall x\in \Omega .
\label{p2}
\end{equation}%
If $\alpha _{i}\leq 0,i=1,2,$ then the constant-sign solutions $%
(u_{+},v_{+}) $ and $(u_{-},v_{-})$ are unique.
\end{theorem}

\begin{proof}
By using sub-supersolutions method, we first prove the existence of a
positive solution $(u_{+},v_{+})$ within $[C^{-1}z_{1},Cy_{1}]\times \lbrack
C^{-1}z_{2},Cy_{2}]$. Given that $z_{i},y_{i}\geq 0$ in $\overline{\Omega }$%
, we have $sgn(u_{+}),sgn(v_{+})\equiv 1$ in $[C^{-1}z_{1},Cy_{1}]$ and $%
[C^{-1}z_{2},Cy_{2}],$ respectively. Obviousely, we have%
\begin{equation*}
-C^{-1}<0\leq \min \{\frac{(C^{-1}z_{1})^{\alpha _{i}}}{(Cy_{2})^{\beta _{i}}%
},\frac{(Cy_{1})^{\alpha _{i}}}{(Cy_{2})^{\beta _{i}}}\},\ \text{for all }%
x\in \Omega _{\delta }.
\end{equation*}%
Moreover, by (\ref{alpha}), (\ref{21}) and for $C>1$, for $\alpha _{i}>0,$
we obtain%
\begin{eqnarray*}
C^{-1}d(x)^{\alpha _{i}-\beta _{i}} &\leq &(Cc)^{-(\alpha _{i}+\beta
_{i})}d(x)^{\alpha _{i}-\beta _{i}} \\
&\leq &\frac{(C^{-1}c^{-1}d(x))^{\alpha _{i}}}{(Ccd(x))^{\beta _{i}}}\leq 
\frac{(C^{-1}z_{1})^{\alpha _{i}}}{(Cy_{2})^{\beta _{i}}},\ \ \text{for all }%
x\in \Omega \backslash \overline{\Omega }_{\delta },
\end{eqnarray*}%
\begin{eqnarray*}
Cd(x)^{\alpha _{i}-\beta _{i}} &\geq &c^{\alpha _{i}+\beta
_{i}}(Cd(x))^{\alpha _{i}-\beta _{i}} \\
&\geq &\frac{(Ccd(x))^{\alpha _{i}}}{(Cc^{-1}d(x))^{\beta _{i}}}\geq \frac{%
(Cy_{1})^{\alpha _{i}}}{(Cy_{2})^{\beta _{i}}},\text{ \ for all }x\in 
\overline{\Omega },
\end{eqnarray*}%
while, if $\alpha _{i}\leq 0,$ it holds%
\begin{equation*}
C^{-1}d(x)^{\alpha _{i}-\beta _{i}}\leq (Ccd(x))^{\alpha _{i}-\beta
_{i}}\leq \frac{(Cy_{1})^{\alpha _{i}}}{(Cy_{2})^{\beta _{i}}},\ \ \text{for
all }x\in \Omega \backslash \overline{\Omega }_{\delta },
\end{equation*}%
\begin{equation*}
Cd(x)^{\alpha _{i}-\beta _{i}}\geq (C^{-1}c^{-1}d(x))^{\alpha _{i}-\beta
_{i}}\geq \frac{(C^{-1}z_{1})^{\alpha _{i}}}{(C^{-1}z_{2})^{\beta _{i}}},%
\text{ \ for all }x\in \overline{\Omega },
\end{equation*}%
provided that $C>1$ is sufficiently large. Then, in view of (\ref{20}) and (%
\ref{22}), it follows that%
\begin{equation*}
\begin{array}{l}
-\Delta (C^{-1}z_{i})+C^{-1}z_{i}=C^{-1}\left\{ 
\begin{array}{ll}
d(x)^{\alpha _{i}-\beta _{i}} & \text{in \ }\Omega \backslash \overline{%
\Omega }_{\delta }, \\ 
-1 & \text{in \ }\Omega _{\delta },%
\end{array}%
\right. \\ 
\\ 
\leq \left\{ 
\begin{array}{cc}
\frac{(C^{-1}z_{1})^{\alpha _{i}}}{(Cy_{2})^{\beta _{i}}} & \text{if }\alpha
_{i}\geq 0 \\ 
\frac{(Cy_{1})^{\alpha _{i}}}{(Cy_{2})^{\beta _{i}}} & \text{if }\alpha
_{i}\leq 0%
\end{array}%
\right. \leq \frac{|u|^{\alpha _{i}}}{|v|^{\beta _{i}}}\text{ in }\Omega ,%
\end{array}%
\end{equation*}%
and 
\begin{equation*}
\begin{array}{l}
-\Delta (Cy_{i})+Cy_{i}=Cd(x)^{\alpha _{i}-\beta _{i}} \\ 
\\ 
\geq \left\{ 
\begin{array}{cc}
\frac{(Cy_{1})^{\alpha _{i}}}{(Cy_{2})^{\beta _{i}}} & \text{if }\alpha
_{i}\geq 0 \\ 
\frac{(C^{-1}z_{1})^{\alpha _{i}}}{(C^{-1}z_{2})^{\beta _{i}}} & \text{if }%
\alpha _{i}\leq 0%
\end{array}%
\right. \geq \frac{|u|^{\alpha _{i}}}{|v|^{\beta _{i}}}\text{ in }\Omega ,%
\end{array}%
\end{equation*}%
for all $(u_{1},u_{2})\in \lbrack C^{-1}z_{1},Cy_{1}]\times \lbrack
C^{-1}z_{2},Cy_{2}].$ Then, \cite[Theorem 2]{KM} ensures the existence a
solution $(u_{+},v_{+})\in \mathcal{C}^{1,\tau }(\overline{\Omega })\times 
\mathcal{C}^{1,\tau }(\overline{\Omega }),$ for certain $\tau \in (0,1),$
for problem $(\mathrm{P})$ within $[C^{-1}z_{1},Cy_{1}]\times \lbrack
C^{-1}z_{2},Cy_{2}]$. In view of Proposition \ref{T5}, $(u_{+},v_{+})$ is a
unique positive solution of $(\mathrm{P})$.

On the other hand, by (\ref{21}) and (\ref{alpha}), for each compact set $%
\mathrm{K}\subset \Omega $, there is a constant $\eta =\eta (\mathrm{K})>0$
such that 
\begin{equation*}
\begin{array}{l}
\eta +\underline{\mathrm{X}}_{i}(x):=\eta +C^{-1}\left\{ 
\begin{array}{ll}
d(x)^{\alpha _{i}-\beta _{i}} & \text{in \ }\Omega \backslash \overline{%
\Omega }_{\delta }, \\ 
-1 & \text{in \ }\Omega _{\delta },%
\end{array}%
\right. \\ 
\\ 
<\min \{(Cc)^{-(\alpha _{i}+\beta _{i})},(Cc)^{\alpha _{i}-\beta _{i}}\}%
\text{ }d(x)^{\alpha _{i}-\beta _{i}} \\ 
\\ 
\leq \left\{ 
\begin{array}{cc}
\frac{(C^{-1}c^{-1}d(x))^{\alpha _{i}}}{(Ccd(x))^{\beta _{i}}} & \text{if }%
\alpha _{i}\geq 0 \\ 
\frac{(Ccd(x))^{\alpha _{i}}}{(Ccd(x))^{\beta _{i}}} & \text{if }\alpha
_{i}\leq 0%
\end{array}%
\right. \leq \left\{ 
\begin{array}{cc}
\frac{(C^{-1}z_{1})^{\alpha _{i}}}{(Cy_{2})^{\beta _{i}}} & \text{if }\alpha
_{i}\geq 0 \\ 
\frac{(Cy_{1})^{\alpha _{i}}}{(Cy_{2})^{\beta _{i}}} & \text{if }\alpha
_{i}\leq 0%
\end{array}%
\right. \\ 
\\ 
\leq \frac{u_{+}^{\alpha _{i}}}{v_{+}^{\beta _{i}}}:=\mathrm{X}_{i}(x)\text{
a.e. in }\Omega \cap \mathrm{K}%
\end{array}%
\end{equation*}%
and%
\begin{equation*}
\begin{array}{l}
\overline{\mathrm{X}}_{i}(x):=Cd(x)^{\alpha _{i}-\beta _{i}}>\eta +\max
\{(Cc)^{\alpha _{i}+\beta _{i}},(C^{-1}c^{-1})^{\alpha _{i}-\beta _{i}}\}%
\text{ }d(x)^{\alpha _{i}-\beta _{i}} \\ 
\\ 
\geq \eta +\left\{ 
\begin{array}{cc}
\frac{(Ccd(x))^{\alpha _{i}}}{(C^{-1}c^{-1}d(x))^{\beta _{i}}} & \text{if }%
\alpha _{i}\geq 0 \\ 
\frac{(C^{-1}z_{1})^{\alpha _{i}}}{(C^{-1}z_{2})^{\beta _{i}}} & \text{if }%
\alpha _{i}\leq 0%
\end{array}%
\right. \geq \eta +\left\{ 
\begin{array}{cc}
\frac{(Cy_{1})^{\alpha _{i}}}{(C^{-1}z_{2})^{\beta _{i}}} & \text{if }\alpha
_{i}\geq 0 \\ 
\frac{(C^{-1}z_{1})^{\alpha _{i}}}{(C^{-1}z_{2})^{\beta _{i}}} & \text{if }%
\alpha _{i}\leq 0%
\end{array}%
\right. \\ 
\\ 
\geq \eta +\frac{u_{+}^{\alpha _{i}}}{v_{+}^{\beta _{i}}}:=\eta +\mathrm{X}%
_{i}(x)\text{ a.e. in }\Omega \cap \mathrm{K},%
\end{array}%
\end{equation*}%
with $\underline{\mathrm{X}}_{i},\mathrm{X}_{i}\mathrm{,}\overline{\mathrm{X}%
}_{i}\in L_{loc}^{\infty }(\Omega )$. By the strong comparison principle 
\cite[Proposition $2.6$]{AR}, we infer that (\ref{p1}) holds true.

Following a quite similar argument as above we obtain the existence a
solution $(u_{-},v_{-})\in \mathcal{C}^{1,\tau }(\overline{\Omega })\times 
\mathcal{C}^{1,\tau }(\overline{\Omega }),$ for certain $\tau \in (0,1),$
for problem $(\mathrm{P})$ satisfying (\ref{p2}), which actually is the
unique negative solution of $(\mathrm{P})$.
\end{proof}

\subsection{Nodal solutions}

\begin{theorem}
Under assumption (\ref{alpha}), problem $(\mathrm{P})$ admits at least a
solution $(u_{\ast },v_{\ast })$ in $\mathcal{H}_{0}^{1}(\Omega )\times 
\mathcal{H}_{0}^{1}(\Omega )$ within $[-C^{-1}z_{1},C^{-1}z_{1}]\times
\lbrack -C^{-1}z_{2},C^{-1}z_{2}].$ Moreover, if $\alpha _{i}\leq 0$, then $%
(u_{\ast },v_{\ast })$ is nodal with synchronous sign components.
\end{theorem}

\begin{proof}
According to Theorem \ref{T4} and Remark \ref{R1}, system $(\mathrm{P})$
admits a nontrivial solution $(u_{\ast },v_{\ast })\in \mathcal{H}%
_{0}^{1}(\Omega )\times \mathcal{H}_{0}^{1}(\Omega )$ located in $%
[-C^{-1}z_{1},C^{-1}z_{1}]\times \lbrack -C^{-1}z_{2},C^{-1}z_{2}]$. In view
of (\ref{p1})-(\ref{p2}), we infer that $(u_{\ast },v_{\ast })$ is a third
solution of $(\mathrm{P})$ while Theorem \ref{T1} together with Remark \ref%
{R1} force that $(u_{\ast },v_{\ast })$ is nodal in the sens that the
components $u_{\ast }$ and $v_{\ast }$ are nontrivial and at least are not
of the same constant sign.

Assume that $u_{\ast }<0<v_{\ast }$. Test the first equation in $(\mathrm{P}%
) $ by $-u_{\ast }^{-}$ we get%
\begin{eqnarray*}
\int_{\Omega }(|\nabla u_{\ast }^{-}|^{2}+|u_{\ast }^{-}|^{2})\text{ }%
\mathrm{d}x &=&-\int_{\Omega }sgn(v_{\ast })\frac{|u_{\ast }|^{\alpha _{1}}}{%
|v_{\ast }|^{\beta _{1}}}u_{\ast }^{-}\text{ }\mathrm{d}x \\
&=&-\int_{\Omega }\frac{|u_{\ast }|^{\alpha _{1}}}{|v_{\ast }|^{\beta _{1}}}%
u_{\ast }^{-}\text{ }\mathrm{d}x<0,
\end{eqnarray*}%
which forces $u_{\ast }^{-}=0$, a contradiction. The same conclusion can be
drawn if we assume $v_{\ast }<0<u_{\ast }$. Hence, $u_{\ast }$ and $v_{\ast
} $ cannot be of opposite constant sign, therefore, at least $u_{\ast }$ or $%
v_{\ast }$ change sign.

Assume that $u_{\ast }$ change sign and $v_{\ast }>0$. Let $\Omega _{\ast
}\subset \Omega $ be a subset such that $u_{\ast }<0$ in $\Omega _{\ast }$.
In view of \cite[Proposition 1.61]{MMPA}, $u_{\ast }^{-}\mathbbm{1}_{\Omega
_{\ast }}\in \mathcal{H}_{0}^{1}(\Omega )$, where $\mathbbm{1}_{\Omega
_{\ast }}$ denotes the characteristic function of $\Omega _{\ast }$. Test
the first equation in $(\mathrm{P})$ by $-u_{\ast }^{-}\mathbbm{1}_{\Omega
_{\ast }}\in \mathcal{H}_{0}^{1}(\Omega )$ we get%
\begin{eqnarray*}
\int_{\Omega _{\ast }}(|\nabla u_{\ast }^{-}|^{2}+|u_{\ast }^{-}|^{2})\text{ 
}\mathrm{d}x &=&-\int_{\Omega _{\ast }}sgn(v_{\ast })\frac{|u_{\ast
}|^{\alpha _{1}}}{|v_{\ast }|^{\beta _{1}}}u_{\ast }^{-}\text{ }\mathrm{d}x
\\
&=&-\int_{\Omega _{\ast }}\frac{|u_{\ast }|^{\alpha _{1}}}{|v_{\ast
}|^{\beta _{1}}}\hat{u}_{1}^{-}\text{ }\mathrm{d}x\leq 0,
\end{eqnarray*}%
which forces $u_{\ast }^{-}\mathbbm{1}_{\Omega _{\ast }}=0$, a
contradiction. The same conclusion can be drawn if we assume $u_{\ast }$
change sign and $v_{\ast }<0$ or $v_{\ast }$ change sign and $u_{\ast }>0$
or $u_{\ast }<0$. Hence, $u^{\ast }$ and $v^{\ast }$ are necessarily
synchronous sign changing satisfying (\ref{40}). This completes the proof.
\end{proof}

\end{document}